\documentclass{birkjour}
\usepackage{hyperref}

 \newtheorem{thm}{Theorem}[section]

 \theoremstyle{definition}
 
 \theoremstyle{remark}

 \numberwithin{equation}{section}

\begin{document}

\title[Bicomplex Lucas and Horadam Numbers]{Bicomplex Lucas and Horadam Numbers}

\author[Serp\.{i}l HALICI]{Serp\.{i}l HALICI}

\address{%
Pamukkale University,\\
Faculty of Arts and Sciences,\\
Department of Mathematics,\\
Denizli/TURKEY}

\email{shalici@pau.edu.tr}

\author[Adnan KARATA\c{S}]{Adnan KARATA\c{S}}
\address{%
Pamukkale University,\\
Faculty of Arts and Sciences,\\
Department of Mathematics,\\
Denizli/TURKEY}
\email{adnank@pau.edu.tr}


\begin{abstract}
In this work, we made a generalization that includes all bicomplex Fibonacci-like numbers such as; Fibonacci, Lucas, Pell, etc.. We named this generalization as bicomplex Horadam numbers. For bicomplex Fibonacci and Lucas numbers we gave some additional identities. Moreover, we have obtained the Binet formula and generating function for bicomplex Horadam numbers for the first time. We have also obtained two important identities that relate the matrix theory to the second order recurrence relations.

\end{abstract}
\maketitle
\section{Introduction}
After making studies for the complex numbers, W. R. Hamilton tried to define three dimensional algebra. Later, he recognized the idea behind the multiplication of imaginary units $i, j, k$ and defined real quaternions in 1844 \cite{hamilton1844}. Thereafter, there is a rising attention for four dimensional algebras. J. Cockle defined the split quaternions in 1848 \cite{cockle}. Inspired by the earlier studies, in 1892 \cite{Segre}, Segre defined bicomplex numbers. Bicomplex numbers neglected because of their zero divisors but now they are attracting attention of researchers.\\
For some basic and important work on bicomplex variables, the interested reader can look at the sources in \cite{banerjee2014fourier, agarwal2015tauberian, wang2013generalized, mursaleen2016approximation}. We should also mention that there are many studies about holomorphic functions which are with bicomplex variables. For some of these, one can look at references \cite{ price1991, luna2015bicomplex}. Also, there are some studies on bicomplex quatum mechanics \cite{rochon2004bicomplex,bender2002complex}, functional analysis with bicomplex scalars \cite{alpay2014basics}, computational studies \cite{toyoshima1998computationally, alfsmann2006families}, etc.
\\\\
Now, let us give definition and some fundamental properties of bicomplex numbers. The set of bicomplex numbers is
$$\mathbb{BC}=\{b=z_1+z_2 j | z_1, z_2 \in \mathbb{C}, j^2=-1  \}.$$
Addition of bicomplex numbers $b_1$ and $b_2$ is component-wise. And the multiplication of their basis elements can be done according to the following table.\\
		\begin{center} \begin{tabular}{ | c | c | c | c |}
	\hline	$.$ &  $i$ &  $j$ &  $k$\\ \hline
		$i$ &  $-1$  &  $k$ & $-j$ \\ \hline
		$j$ & $k$ &  $-1$  &  $-i$ \\ \hline
		$k$ &  $-j$ & $-i$ & $ 1$  \\  \hline \end{tabular} \end{center}
$$\mbox{\scriptsize Table.1 Multiplication of bicomplex numbers}$$
\\ where $ij=ji=k$. It should be noted that the multiplication of bicomplex numbers is similar to multiplication of real quaternions.  If we want to say some differences between these two sets of numbers, we can list them as bicomplex numbers are commutative, have zero divisors and non-trivial idempotent elements, but real quaternions are non-commutative and don't have zero divisors and non-trivial idempotent elements. That is, the following equations are satisfied for elements of the set $\mathbb{BC}$.
$$ij=ji=k,$$
$$(i+j)(i-j)=i^2-ij+ji-j^2=-1-k+k+1=0$$
and
$$\Big( \frac{1+k}{2} \Big) ^2=\frac{1+k}{2}.$$
Apart from these properties, bicomplex numbers have three different involutions and norms. For any bicomplex number $b$  these involutions are as follows.
$$b=a_1 + a_2 i + a_3 j + a_4 k,$$
$$\overline{b}_i=a_1 - a_2 i + a_3 j - a_4 k,$$
$$\overline{b}_j=a_1 + a_2 i - a_3 j - a_4 k,$$
$$\overline{b}_k=a_1 - a_2 i - a_3 j + a_4 k.$$ \\
Norms arising from the definitions of involutions are as follows.\\
$$N_i(b)=b \overline{b}_i=(a_1 + a_2 i + a_3 j + a_4 k) (a_1 - a_2 i + a_3 j - a_4 k),$$
$$N_j(b)=b \overline{b}_j=(a_1 + a_2 i + a_3 j + a_4 k) (a_1 + a_2 i - a_3 j - a_4 k),$$
$$N_k(b)=b \overline{b}_k=(a_1 + a_2 i + a_3 j + a_4 k) (a_1 - a_2 i - a_3 j + a_4 k).$$ \\
We should note all of these norms are isotropic. And for $N_i(\,\,)$ we can take $(1+k)$ and calculate the norm.
$$N_i(1+k)=(1+k)(1-k)=1^2-k+k-k^2=0.$$
Up to now we presented some fundamental properties of bicomplex numbers. In the following sections firstly we recall the definitions of bicomplex Fibonacci and Lucas numbers which is defined in the reference\cite{Nurkan}. Secondly, we give some important additional identities for the recurrence relations involving these numbers. And then, we define the bicomplex Horadam numbers which generalizes all Fibonacci-like recurrence relations on bicomplex numbers. In addition we give some fundamental identities for bicomplex Horadam numbers. Finally, in conclusion we summarize the paper and give a insight for further studies.
%
%

\section{Bicomplex Fibonacci and Lucas Numbers}

In \cite{Nurkan}, bicomplex Fibonacci and Lucas numbers and their recurrence relations are studied. Let $BF_n$ and $BL_n$ be nth bicomplex Fibonacci and nth bicomplex Lucas numbers, respectively.
$$BF_n=f_n+f_{n+1}i+f_{n+2}j+f_{n+3}k, \,\,\,\, BF_{n+2}=BF_{n+1}+BF_{n}$$ and
$$BL_n=k_n+k_{n+1}i+k_{n+2}j+k_{n+3}k, \,\,\,\, BL_{n+2}=BL_{n+1}+BL_{n}$$
where $f_n$ and $k_n$ are the classical Fibonacci and Lucas numbers, respectively. Finding any element that provides recurrence relations is an important problem. The most convenient method to overcome this problem is to use the Binet formula or generating function. In \cite{Nurkan}, Binet formulas for bicomplex Fibonacci and Lucas numbers are given as follows, respectively.
\begin{equation} \label{Binf} BF_n=\frac{\underline{\alpha} \alpha^n - \underline{\beta} \beta^n}{\alpha - \beta} , \,\ \underline{\alpha}=1+\alpha i+\alpha^2 j+\alpha^3 k\end{equation} and
\begin{equation} \label{Binl} BL_n=\underline{\alpha} \alpha^n + \underline{\beta} \beta^n , \,\ \underline{\beta}=1+\beta i+\beta^2 j+\beta^3 k \end{equation}\\
where $\alpha$ and $\beta$ are roots of characteristic equation of the classical Fibonacci and Lucas numbers. We should especially mention here that the authors' Binet formula is identical to the formula for Fibonacci quaternions given by Halici in \cite{Hal1}.
In this section of the study, especially we will focus on bicomplex Lucas numbers and present some fundamental identities and equalities related with both bicomplex Fibonacci and Lucas numbers. Our first theorem is about the generating function of bicomplex Lucas numbers.
\begin{thm}
Generating function for bicomplex Lucas numbers is
\begin{equation} \label{eq:1} g(t)=\frac{2+i+3j+4k + (-1+2i+j+3k)t}{1-t-t^2}. \end{equation}
\end{thm}
\begin{proof}
Let $g(t)$ be the generating function for bicomplex Lucas numbers as
$$g(t)=\sum^{\infty}_{n=0} BL_n t^n.$$
Using $g(t)$, $tg(t)$ and $t^2g(t)$ we get the following equations.
$$tg(t)=\sum^{\infty}_{n=0} BL_n t^{n+1}, \,\,\,\, t^2 g(t)=\sum^{\infty}_{n=0} BL_n t^{n+2}.$$
After the needed calculations the generating function for bicomplex Lucas numbers is obtained as
$$g(t)=\frac{BL_0 + (BL_1-BL_0)t}{1-t-t^2}.$$
Thus, the equation (\ref{eq:1}) is obtained.
\end{proof}
Catalan identity is one of the famous identities which is generalized form of Cassini identity. In the next theorem, we will present Catalan identity for bicomplex Lucas numbers.
\begin{thm}
Catalan identity for bicomplex Lucas numbers is
\begin{equation} \label{eq:2} BL_n^2 - BL_{n+r} BL_{n-r}=15(-1)^{n-r} f_r^2 (2j+k). \end{equation}
\end{thm}
\begin{proof}
Using the Binet formula for bicomplex Lucas numbers (\ref{Binl})  and using the properties of $\alpha$ and $\beta$, we obtained the desired result;
$$BL_n^2 - BL_{n+r} BL_{n-r}=(\underline{\alpha} \alpha^n + \underline{\beta} \beta^n)^2-(\underline{\alpha} \alpha^{n+r} + \underline{\beta} \beta^{n+r})( \underline{\alpha} \alpha^{n-r} + \underline{\beta} \beta^{n-r})$$
$$=-\underline{\alpha}\underline{\beta} (\alpha \beta)^{n-r}(\alpha^{2r}-2(\alpha \beta)^{r}+\beta^{2r})$$
$$BL_n^2 - BL_{n+r} BL_{n-r}=15(-1)^{n-r} f_r^2 (2j+k)$$
\end{proof}
In \cite{Nurkan}, Cassini identity for bicomplex Lucas numbers is given in theorem $4$ with little mistake. So, the Cassini identity for bicomplex Lucas numbers can be obtained from equation (\ref{eq:2}) as
$$BL_n^2 - BL_{n+1} BL_{n-1}=15(-1)^{n-1}(2j+k).$$
In the next theorem we will derive various identities including bicomplex Fibonacci and Lucas numbers.
\begin{thm}
For bicomplex Fibonacci and Lucas numbers the following equations are satisfied.
\begin{equation} \label{eq:3} BF_n + BL_n = 2 BF_{n+1}   \end{equation}
\begin{equation} \label{eq:4} BL_{n-1} + BL_{n+1} = 5 BF_{n}   \end{equation}
\begin{equation} \label{eq:5} \sum ^n _{i=1} BL_{2i-1} = BL_{2n} - BL_{0}   \end{equation}
\begin{equation} \label{eq:6} \sum ^n _{i=1} BL_{i} = BL_{n+2} - BL_{2} \end{equation}
\begin{equation} \label{eq:7} \sum ^n _{i=1} BL_{2i} = BL_{2n+1} - BL_{1} \end{equation}
\begin{equation} \label{eq:8} 5 BF_{n}^2 - BL_{n} ^2 = 12 (-1)^{n+1} (2j+k) \end{equation}
\begin{equation} \label{eq:9} \sum ^n _{i=1} BL_{i}^2 = BL_n BL_{n+1} - (5 - 10i - 12j + 9k) \end{equation}
 $$BL_{n} BL_{m} + BL_{n+1} BL_{m+1}=$$
\begin{equation} \label{eq:10}  5(2 BF_{n+m+1}+2f_{n+m+4}-f_{n+m+1}-2f_{n+m+6}i-2f_{n+m+5}j+2f_{n+m+4}ij)   \end{equation}
\end{thm}
\begin{proof}
It is only enough to show that equation (\ref{eq:10}) is true because the other equations can be easily seen by the help of the Fibonacci  and Lucas identities that exist in the reference \cite{koshy}.
We note that, for equation (\ref{eq:10}), firstly we need to use the identity $$k_n k_m + k_{n+1} k_{m+1} = 5f_{m+n+1}$$  where $k_n$ is the $n th$ Lucas number. So, let us explicitly calculate the right hand side of the equation (\ref{eq:10}).
$$=[(k_n k_m + k_{n+1} k_{m+1}) - (k_{n+1} k_{m+1} + k_{n+2} k_{m+2})$$
$$-(k_{n+2} k_{m+2} + k_{n+3} k_{m+3})+(k_{n+3} k_{m+3} + k_{n+4} k_{m+4})]$$
$$ +[(k_{n} k_{m+1} + k_{n+1} k_{m})-(k_{n+2} k_{m+3} + k_{n+3} k_{m+2})$$
$$+(k_{n+1} k_{m+2} + k_{n+2} k_{m+1})-(k_{n+3} k_{m+4} + k_{n+4} k_{m+3})]i$$
$$ +[(k_{n} k_{m+2} + k_{n+2} k_{m})-(k_{n+1} k_{m+3} + k_{n+3} k_{m+1})$$
$$+(k_{n+1} k_{m+3} + k_{n+3} k_{m+1})-(k_{n+2} k_{m+4} + k_{n+4} k_{m+2})]j$$
$$+[(k_{n} k_{m+3} + k_{n+3} k_{m})+(k_{n+1} k_{m+2} + k_{n+2} k_{m+1})$$
$$+(k_{n+1} k_{m+4} + k_{n+4} k_{m+1})+(k_{n+2} k_{m+3} + k_{n+3} k_{m+2})]ij.$$
Using the identity $$k_n k_m + k_{n+1} k_{m+1} = 5f_{m+n+1},$$  we have
$$=5(f_{m+n+1}-f_{m+n+3}-f_{m+n+5}+f_{m+n+7})$$
$$+5(f_{m+n+2}+f_{m+n+2}-f_{m+n+6}-f_{m+n+6})i$$
$$+5(f_{m+n+3}+f_{m+n+3}-f_{m+n+5}-f_{m+n+5})j$$
$$+5(f_{m+n+4}+f_{m+n+4}+f_{m+n+4}+f_{m+n+4})ij.$$ \\
Then, making the necessary calculations we get the following result.\\
$$5(2 BF_{n+m+1}+2f_{n+m+4}-f_{n+m+1}-2f_{n+m+6}i-2f_{n+m+5}j+2f_{n+m+4}ij)$$ \\
which is desired. In here, we used the letters $ij$ instead of $k$ so that there would be no confusion.
\end{proof}
Up to now, we give basic properties and some advanced identities including bicomplex Fibonnaci and Lucas numbers. In the next section, we present Horadam numbers which is defined to generalize all second order linear recurrence relations.
In addition, we define bicomplex numbers with coefficients from Horadam sequence and name it as bicomplex Horadam numbers. Furthermore, we show that our definition generalize all of former studies related with bicomplex Horadam numbers.
%
%
%
%
\section{Bicomplex Horadam Numbers}
In \cite{Hor3}, Horadam defined the Horadam numbers as following
\begin{equation} \label{eq:11} w_n = w_n (a, b; p, q) = p w_{n-1} + q w_{n-2}; \,\  n\geq 2,\,\ W_0 = a, \,\ W_1 = b \end{equation}
where $a, b$ and $p, q \in \mathbb{Z}$. There are some papers which are dedicated to generalization of Fibonacci numbers defined on different algebras. Some of them are \cite{Hor3, Hor4, swamy1973, Cris1, halici2013complex, Adn2017, halici2019bicomplex}. Now, using by the Horadam numbers let us define nth bicomplex Horadam number as follows.
\begin{equation} \label{eq:12} BH_n=w_{n} + w_{n+1} i + w_{n+2} j + w_{n+3} k   \end{equation}
For bicomplex Horadam numbers, the recurrence relation and first two elements of these numbers as follows.\\
\begin{equation} \label{eq:13} BH_{n+2}=p BH_{n+1} + q BH_{n}, \end{equation} and
\begin{equation} \label{eq:14} BH_0 = a+ b i + ( pb+qa ) j + (p^2b +pqa + qb ) k   \end{equation}
\begin{equation} \label{eq:15}  BH_1 = b + ( pb+qa ) i + (p^2b +pqa + qb ) j + (p^3b+p^2qa+2pqb+q^2a) k,   \end{equation}\\  respectively.
If we choose values $a=0, b=p=q=1$ and $a=2, b =p = q = 1$ in the above equations, then we get the bicomplex Fibonacci and Lucas numbers' recurrence relation and first two elements, respectively.
Now, let us give Binet formula for bicomplex Horadam numbers and show the formula satisfies the equations (\ref{Binf}) and (\ref{Binl}) for appropriate initial values.
\begin{thm} Binet formula for nth bicomplex Horadam number is\\
\begin{equation} \label{eq:16} BH_{n}=\frac{A\underline{\alpha} \alpha^n - B\underline{\beta}\beta^n}{\alpha - \beta} \end{equation} \\
where\\\\  $ A=b-a \beta, \,\  B=b-a \alpha$ and $\underline{\alpha}=1+\alpha i + \alpha^2 j + \alpha^3 k, \underline{\beta}=1+\beta i + \beta^2 j + \beta^3 k$.
\end{thm}
\begin{proof}
Using the roots of characteristic equation $t^2-pt-q=0$ related to Horadam numbers that is, \\
  $$\alpha = \frac{p+\sqrt{p^2 + 4q}}{2},  \beta = \frac{p-\sqrt{p^2 + 4q}}{2}$$ 
and by doing the necessary calculations we obtain the following equation;
$$BH_{n}=\frac{A\underline{\alpha} \alpha^n - B\underline{\beta}\beta^n}{\alpha - \beta}$$
which is desired.
\end{proof}
If you are careful, the equation (\ref{eq:16}) is a generalization of Binet formula for bicomplex Fibonacci and Lucas numbers. In fact, if we choose the values $a, b$ and $p, q$ as $0, 1 $ and $1, 1$  respectively, then we have the equation (\ref{Binf}). Also, by typing the values $a=2, b=1$ and $ p= 1, q=1$  the equation  (\ref{Binl}) is obtained.
\begin{thm}
Generating function for Bicomplex Horadam numbers is
\begin{equation} \label{generating}
g(t)=\frac{BH_0+(BH_1 - p BH_0)t}{1-pt-q t^2}.
\end{equation}
\end{thm}
\begin{proof}
First, let us write the generating function for these numbers as follows
$$g(t)=\sum ^{\infty} _{n=0} BH_{n} t^n .$$ Then, to find the desired form of the $g(t)$ let us calculate the equations for $ptg(t)$ and $qt^2g(t)$;
\begin{equation}
ptg(t)=\sum ^{\infty} _{n=0} pBH_{n} t^{n+1} \,\,\,\,\, and \,\,\,\,\,  qt^2g(t)=\sum ^{\infty} _{n=0} qBH_{n} t^{n+2}
\end{equation}
Also, if we take advantage of the characteristic equation of these numbers, $t^2-pt-q=0$, we find an explicit form for the generating function $g(t)$ as follows.
$$g(t)=\frac{BH_0+(BH_1 - p BH_0)t}{1-pt-q t^2}$$
which is desired. Thus, the proof is completed.
\end{proof}
Note that if the initial values of Fibonacci numbers are used in the above equation (\ref{generating}), the formula of generating function for bicomplex Fibonacci numbers is obtained.
In particular, it is not difficult to see that the following formula can be obtained if the initial values of Lucas numbers are to be selected $$\frac{BL_0+(BL_1 - BL_0)t}{1-t-t^2}.$$
\\\\
Now in the following theorem, we give the identity which is an important generalization of Cassini identity and called as Catalan identity.
So, we will see later that we can get the Catalan identity for all bicomplex Fibonacci-like numbers using the formula we mentioned.
\begin{thm}
The Catalan identity for bicomplex Horadam numbers is
\begin{equation} \label{eq:17}
BH^2_n - BH_{n+r}BH_{n-r}=\frac{AB \, \underline{\alpha}\underline{\beta} \, (-q)^{n-r}}{p^2+4q} (\alpha^r - \beta^r)^2.
\end{equation}
\end{thm}
\begin{proof} Let us use the Binet formula which we found for Horadam numbers on the Catalan identity. Therefore, we can write an equation such as
$$\Big( \frac{A\underline{\alpha} \alpha^n - B\underline{\beta}\beta^n}{\alpha - \beta} \Big) ^2-\frac{A\underline{\alpha} \alpha^{n+r} - B\underline{\beta}\beta^{n+r}}{\alpha - \beta} \frac{A\underline{\alpha} \alpha^{n-r} - B\underline{\beta}\beta^{n-r}}{\alpha - \beta},$$
If we make some adjustments and calculations in the last equation above, then we have the following formula.
$$\frac{1}{(\alpha-\beta)^2} (AB \, \underline{\alpha}\underline{\beta} \, (-q)^{n-r})(\alpha^r - \beta^r)^2.$$
With the help of the properties $\alpha$ and $\beta$, we get the following equality.
$$ BH^2_n - BH_{n+r}BH_{n-r}= \frac{AB \, \underline{\alpha}\underline{\beta} \, (-q)^{n-r}}{p^2+4q} (\alpha^r - \beta^r)^2 .$$
Thus, the proof is completed.
\end{proof}
Next we will give an important equality which is a special case of the Catalan identity and relates the matrix theory to the recurrence relations.
\begin{thm}
The Cassini identity for bicomplex Horadam numbers is
\begin{equation}
BH^2_n - BH_{n+1}BH_{n-1}={AB \, \underline{\alpha}\underline{\beta} \, (-q)^{n-1}}.
\end{equation}
\end{thm}
\begin{proof}
The proof of this theorem can easily be seen as an oversimplification of the catalan equation. For this purpose, in equation (\ref{eq:17}), it is enough to write $r=1$. As a result, it can be seen that the Catalan identity is generalized to the Cassini identity.
\end{proof}

\section{Conclusion}
In this work we have done, we first described the bicomplex numbers with coefficients from the Fibonacci and Lucas sequences. We have given many equations that hold an important place in the literature on these numbers. We then used the Horadam numbers to define a new set of numbers that generalizes all these types of numbers, and we call it bicomplex Horadam numbers. We have also gave some important identities for these numbers. For further studies, we plan to find some additional identities and properties for these new numbers.
\bibliographystyle{acm}
\bibliography{mybibfile}
\end{document}